\documentclass[preprint,12pt]{elsarticle}
\usepackage{amsthm,amsmath,amssymb}
\usepackage[colorlinks=true,citecolor=black,linkcolor=black,urlcolor=blue]{hyperref}
\usepackage{graphicx}
\usepackage{float}
\usepackage{mathtools}
\usepackage{epstopdf}
\usepackage{epsfig}
\usepackage{subfigure}
\usepackage{relsize}
\allowdisplaybreaks[4]

\theoremstyle{plain}
\newtheorem{theorem}{Theorem}
\newtheorem{lemma}[theorem]{Lemma}

\newtheorem{proposition}[theorem]{Proposition}

\theoremstyle{definition}
\newtheorem{definition}[theorem]{Definition}

\newtheorem{problem}[theorem]{Problem}

\theoremstyle{remark}

\title{Partial Petrial polynomials for complete graphs and paths}
\author{
Qi Yan\footnote{Corresponding author.}\\
\small School of Mathematics and Statistics\\[-0.8ex]
\small Lanzhou University\\[-0.8ex]
\small P. R. China\\
Yuancheng Li\\
\small School of Mathematics (Zhuhai)\\[-0.8ex]
\small Sun Yat-sen University\\[-0.8ex]
\small P. R. China\\
\small{\tt Email:yanq@lzu.edu.cn; liycheng2022@163.com}
}
\date{}
\journal{Discrete Applied Mathematics}

\begin{document}
\begin{abstract}
Recently, Gross, Mansour, and Tucker introduced the partial Petrial polynomial, which enumerates all partial Petrials of a ribbon graph by Euler genus. They provided formulas or recursions for various families of ribbon graphs, including ladder ribbon graphs.
In this paper, we focus on the partial Petrial polynomial of bouquets, which are ribbon graphs with exactly one vertex. We prove that the partial Petrial polynomial of a bouquet primarily depends on its intersection graph, meaning that two bouquets with identical intersection graphs will have the same partial Petrial polynomial. Additionally, we introduce the concept of the partial Petrial polynomial for circle graphs and prove that for a connected graph with $n$ vertices ($n\geq 2$), the polynomial has non-zero coefficients for all terms of degrees from 1 to $n$ if and only if the graph is complete. Finally, we present the partial Petrial polynomials for paths.
\end{abstract}
\maketitle

\section{Introduction}
In 1979, Wilson \cite{SEW} introduced the concept of the  Petrial. For an embedded graph
$G$, the Petrial is constructed using the same vertices and edges as $G$, but with the faces replaced by Petrie polygons. These polygons are formed by closed left-right walks in $G$. Petrials are particularly easy to define in the context of ribbon graphs. Specifically, the Petrial of a ribbon graph $G$ is obtained by detaching one end of each edge from its incident vertex disc, giving the edge a half-twist, and then reattaching it to the vertex disc.

It is important to note that the Petrial arises from a local operation applied to each edge of a ribbon graph. When this operation is applied to subsets of the edge set, it results in a partial Petrial.
\begin{definition}[\cite{EM}]
Let $G$ be a ribbon graph and $A\subseteq E(G)$. Then the \emph{partial Petrial}, $G^{\times|A}$, of a ribbon graph $G$ with respect to $A$ is the ribbon graph obtained from $G$ by adding a half-twist to each edges in $A$.
\end{definition}
Along with other partial dualities \cite{EM1}, partial Petrials have garnered increasing attention due to their applications in various fields, including topological graph theory, knot theory, matroids, delta-matroids, and physics. Recently, Gross, Mansour, and Tucker introduced the partial Petrial polynomial for arbitrary ribbon graphs.
\begin{definition}[\cite{GMT2}]
The \emph{partial Petrial polynomial} of any ribbon graph $G$ is the generating function
$$^{\partial}\varepsilon^{\times}_{G}(z)=\sum_{A\subseteq E(G)}z^{\varepsilon(G^{\times|A})}$$
that enumerates all partial Petrials of $G$ by Euler genus.
\end{definition}
Formulas or recursions are given for various families of ribbon graphs, including ladder ribbon graphs \cite{GMT2}.  A \emph{bouquet} is a ribbon graph with exactly one vertex.  We say that two loops in a bouquet are \emph{interlaced} if their ends are met in an alternating order when travelling round the vertex boundary. The {\it intersection graph} \cite{CSL} $I(B)$ of  a bouquet $B$ is the graph with vertex set $E(B)$, and in which two vertices $e_1$ and $e_2$ of $I(B)$ are adjacent if and only if $e_1$ and $e_2$ are interlaced in $B$.

In this paper, we prove that the partial Petrial polynomial of a bouquet primarily depends  on its intersection graph, rather than the bouquet itself. Specifically, two bouquets with identical intersection graphs have the same partial Petrial polynomial. We also introduce the concept of the partial Petrial polynomial for circle graphs. Furthermore, we prove that for a connected graph $G$ with $n$ vertices ($n\geq 2$), the partial Petrial polynomial of $G$ has non-zero coefficients for all terms of degrees ranging from 1 to $n$ if and only if
$G$ is a complete graph.  Finally, we present the partial Petrial polynomials for paths.

\section{Partial Petrial polynomial of circle graphs}

\begin{definition}[\cite{BR}]
A {\it ribbon graph} $G$ is a (orientable or non-orientable) surface with boundary,
represented as the union of two sets of topological discs, a set $V(G)$ of vertices, and a set $E(G)$ of edges, subject to the following restrictions.
\begin{enumerate}
\item[(1)] The vertices and edges intersect in disjoint line segments, we call them {\it common line segments} as in \cite{Metrose}.
\item[(2)] Each such common line segment lies on the boundary of precisely one vertex and precisely one edge.
\item[(3)] Every edge contains exactly two such common line segments.
\end{enumerate}
\end{definition}

If $G$ is a ribbon graph, we denote by $f(G)$ the number of boundary components of $G$, and we define $v(G)$, $e(G)$, and $c(G)$ to be the number of vertices, edges, and connected components of $G$, respectively. We let \[\chi(G)=v(G)-e(G)+f(G),\] the usual \emph{Euler characteristic}, where $G$ is connected or not. The notation $\varepsilon(G)$ represents the \emph{Euler genus} of $G$, that is, \[\varepsilon(G)=2c(G)-\chi(G).\]

Let $G$ be a ribbon graph, with $v\in V(G)$ and $e\in E(G)$.  We denote by $d{(v)}$ the degree of the vertex $v$ in $G$, which is the number of half-edges incident to $v$.
By removing the common line segments from the boundary of $v$,
we obtain $d(v)$ disjoint line segments, denoted as $vl_1, vl_2, \ldots, vl_{d(v)}$, referred to as \emph{vertex line segments} \cite{Metrose}.
Similarly, by removing the common line segments from the boundary of $e$, we obtain two disjoint line segments, termed \emph{edge line segments} \cite{Metrose}, denoted as $\{el_1, el_2\}$. See Figure \ref{Fig05} for an illustration.

\begin{figure}
\begin{center}
\includegraphics[width=9cm]{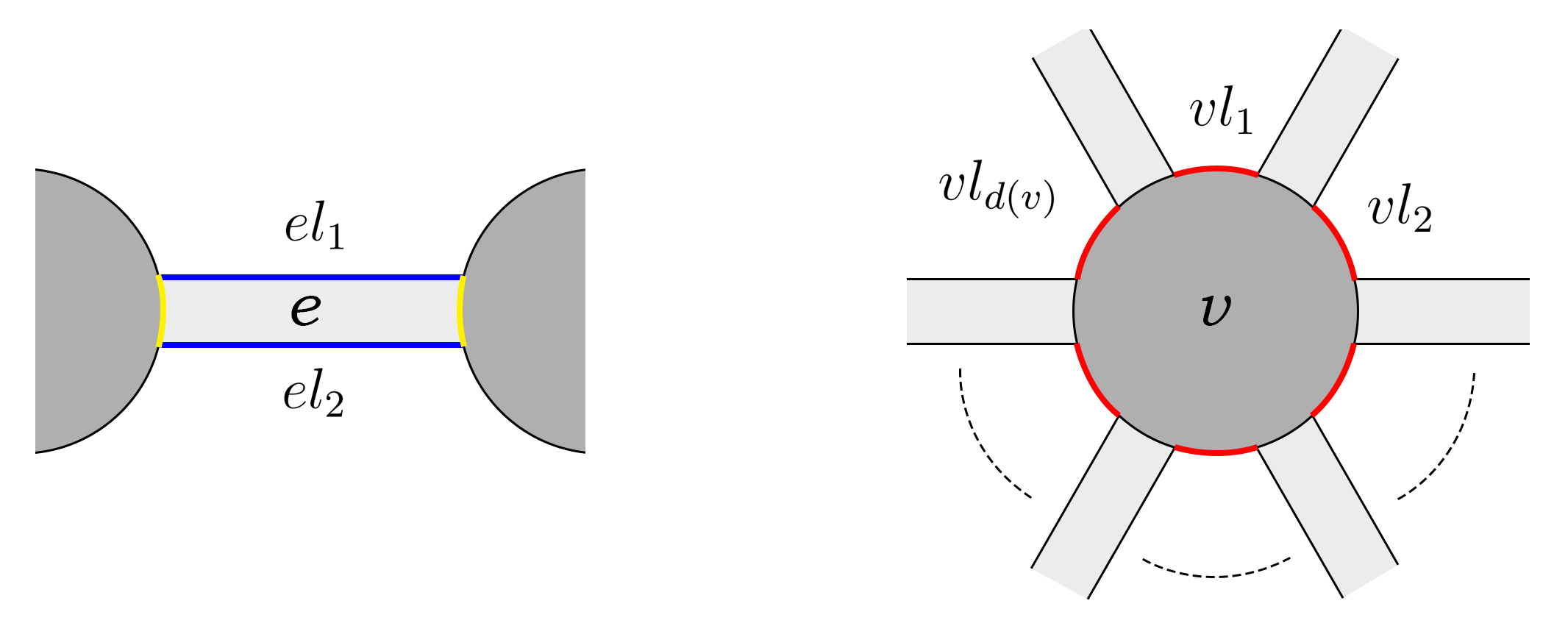}
\caption{Common line segments (yellow), edge line segments
(blue) and vertex line segments (red).}
\label{Fig05}
\end{center}
\end{figure}

Moffatt \cite{Moffatt15} defined the \emph{one-vertex-joint} operation on two disjoint ribbon graphs $P$ and $Q$, denoted by $P\vee Q$, in two steps:
\begin{description}
  \item[(i)] Choose an arc $p$ on the boundary of a vertex-disc $v_1$ of $P$ that lies between two consecutive ribbon ends, and choose another such arc $q$ on the boundary of a vertex-disc $v_2$ of $Q$.
  \item[(ii)] Paste the vertex-discs $v_1$ and $v_2$ together by identifying the arcs $p$ and $q$.
\end{description}
Note that \[f(P\vee Q)=f(P)+f(Q)-1.\]

We say that a ribbon graph $G$ is \emph{prime}, if there do not exist non-empty ribbon subgraphs $G_1, \ldots, G_k$ of $G$ (with $k\geq 2$) such that $G=G_1\vee \ldots\vee G_k$.
Clearly, a bouquet $B$ is prime if and only if its intersection graph $I(B)$ is connected.

A \emph{signed rotation} \cite{GT} of a bouquet $B$ is a cyclic ordering of the half-edges at the vertex, and if the edge is an orientable loop, then we give the same sign $+$  or $-$ to the corresponding two half-edges. Otherwise, we give different signs (one $+$, the other $-$). The sign $+$ is always omitted. Reversing the signs of both half-edges with a given loop yields an equivalent signed rotation. Note that the signed rotation of a bouquet is independent of the choice of the first half-edge. Sometimes we will use the signed rotation to represent the bouquet itself. In terms of signed rotation, $B^{\times|A}$ is formed by toggling one of the signs of every loop in $A$.

The {\it signed intersection graph} $SI(B)$ of a bouquet $B$ consists of $I(B)$ and a $+$ or $-$ sign at each vertex of $I(B)$, where the vertex corresponding to the orientable loop of $B$ is signed $+$ and the vertex corresponding to the non-orientable loop of $B$ is signed $-$.

\begin{proposition}[\cite{QYJ3}]\label{pro-04}
If two bouquets $B_1$ and $B_2$ have the same signed intersection graph, then \[{^\partial\varepsilon^{\times}_{B_1}(z)}={^\partial\varepsilon^{\times}_{B_2}(z)}.\]
\end{proposition}

\begin{lemma}\label{lem-06}
Let $G$ be a ribbon graph and $A\subseteq E(G)$. Then \[{^\partial\varepsilon^{\times}_{G}(z)}={^\partial\varepsilon^{\times}_{G^{\times|A}}(z)}.\]
\end{lemma}
\begin{proof}
This is because the sets of all partial Petrials of $G$ and $G^{\times|A}$ are the same.
\end{proof}

\begin{theorem}\label{the-01}
If two bouquets $B_1$ and $B_2$ have the same intersection graph, then \[{^\partial\varepsilon^{\times}_{B_1}(z)}={^\partial\varepsilon^{\times}_{B_2}(z)}.\]
\end{theorem}
\begin{proof}
Since $B_1$ and $B_2$ have the same intersection graph, there exists a subset $A$ of vertices in $SI(B_1)$ such that each vertex in $A$ corresponds to a vertex in $SI(B_2)$ with a different sign. We also denote its corresponding edge subset of $B_1$ by $A$.
Thus, $SI(B_1^{\times|A})=SI(B_2)$. By Proposition \ref{pro-04}, we have \[{^\partial\varepsilon^{\times}_{B_1^{\times|A}}(z)}={^\partial\varepsilon^{\times}_{B_2}(z)}.\]
Furthermore, by Lemma \ref{lem-06}, \[{^\partial\varepsilon^{\times}_{B_1}(z)}={^\partial\varepsilon^{\times}_{B_1^{\times|A}}(z)}
={^\partial\varepsilon^{\times}_{B_2}(z)}.\]
\end{proof}

A graph is a \emph{circle graph} if it is the intersection graph of a bouquet.

\begin{definition}\label{def-01}
The \emph{partial Petrial polynomial}, denoted by $P_{G}^{\times}(z)$, of a circle graph $G$ is defined as $P_{G}^{\times}(z):={^\partial\varepsilon^{\times}_{B}(z)}$, where $B$ is a bouquet such that $G=I(B)$.
\end{definition}

The well-definedness of Definition \ref{def-01} is ensured by Theorem \ref{the-01}.

\section{Partial Petrial polynomials for complete graphs}


Suppose $P=p_{1}p_{2}\cdots p_{k}$ is a \emph{string}. Then we call
$P^{-1}=(-p_{k})\cdots (-p_{2})(-p_{1})$  the \emph{inverse} of $P$.
Suppose $P$ and $Q$ are strings. In \cite{QYJ}, Yan and Jin introduced four operations on signed rotations as follows:

\begin{description}
  \item[Operation 1] Change the signed rotation from $PaQba$ to $PabQa$.
  \item[Operation 2] Change the signed rotation from $P(-a)Qba$ to $P(-b)(-a)Qa$.
  \item[Operation 3] Simplify the signed rotation from $Pabab$ to $P$.
  \item[Operation 4] Simplify the signed rotation from $P(-a)a$ to $P$.
\end{description}

\begin{proposition}[\cite{QYJ}]\label{pro03}
The operations 1, 2, 3 and 4 do not alter the number of boundary components.
\end{proposition}

\begin{proposition}\label{pro02}
Let $G$ be a connected ribbon graph. Then the highest degree of $^\partial\varepsilon^{\times}_{G}(z)$ is $e(G)-v(G)+1$.
\end{proposition}
\begin{proof}
For any subset $A\subseteq E(G)$, by Euler's formula, we have \[\varepsilon(G^{\times|A})=e(G^{\times|A})+2-v(G^{\times|A})-f(G^{\times|A}).\]
Since partial Petrial does not alter the underlying graph, we have $v(G)=v(G^{\times|A})$ and $e(G)=e(G^{\times|A})$. Since $f(G^{\times|A})\geq1$, we can derive \[\varepsilon(G^{\times|A})=e(G)+2-v(G)-f(G^{\times|A})\leq e(G)-v(G)+1.\]
Hence, the degree of $^\partial\varepsilon^{\times}_{G}(z)$ is at most $e(G)-v(G)+1$.

Next, we prove that the maximum degree of $^\partial\varepsilon^{\times}_{G}(z)$ can indeed reach this upper bound. It suffices to show that there exists an $A\subseteq E(G)$ such that $f(G^{\times|A})=1$. If $f(G)=1$, then $A$ can be taken as the empty set. Otherwise, if $f(G)>1$, then there exists an edge ribbon $e\in E(G)$ such that the two edge line segments of $e$ are contained in different boundary components $\alpha_1$ and $\alpha_2$ of $G$. Therefore, these two boundary components will merge into a new boundary component in $G^{\times|e}$, while the other boundary components of $G$ and $G^{\times|e}$ remain in one-to-one correspondence. Hence, $f(G^{\times|e})=f(G)-1$. By repeatedly performing similar operations, we can reduce the number of boundary components until it becomes 1. The collection of edge ribbons formed by each step constitutes the set $A$.

\end{proof}

\begin{proposition}[\cite{GMT2}]\label{pro01}
For any ribbon graph $G$, the partial Petrial polynomial $^\partial\varepsilon^{\times}_{G}(z)$ is interpolating.
\end{proposition}

\begin{lemma}\label{lem03}
Let $B$ be a prime bouquet with $e(B) = n$, where $n \geq 2$. Then $f(B)\leq n$.
Moreover, $f(B)=n$ if and only if \[B=[e_1, e_2, \ldots, e_n, -e_1, -e_2, \ldots, -e_n].\]
\end{lemma}

\begin{proof}
The boundary components of $B$ are formed by closed curves consisting of alternating edge line segments and vertex line segments. Suppose there exists an edge $e$ and a vertex line segment $vl_i$ of $B$ such that either $el_1$-$vl_i$ or $el_2$-$vl_i$ forms a boundary component of $B$. Then $B$ would take the form $[e, e, M]$, where $M$ is a non-empty string due to $e(B) \geq 2$. This contradicts the condition that $B$ is prime. Therefore, each boundary component of $B$ must contain at least two edge line segments and two vertex line segments. Since each edge line segment of $B$ appears in exactly one boundary component, and the total number of edge line segments in $B$ is $2n$, it follows that $f(B)\leq n$. Furthermore, $f(B)=n$ if and only if each boundary component is composed of exactly two edge line segments and two vertex line segments. We now analyze this specific case.

Since $B$ is prime and $e(B) \geq 2$, we can initially represent $B$ as $[e_1, e_2, \ldots]$, where additional half-edges and their signs will be determined. Consider a boundary $\alpha$ formed by one edge line segment from $e_1$, one from $e_2$, and two vertex line segments.
By considering possible configurations, we claim that:

\begin{description}
  \item[Claim 1.] $e_1$ and $e_2$ are interlaced.
  \item[Proof of Claim 1.] Suppose $B=[e_1, e_2, P, \pm e_2, R, \pm e_1, Q]$, where $\pm$ indicates indeterminate signs.
\begin{figure}[htbp]
  \centering
  \includegraphics[width=14cm]{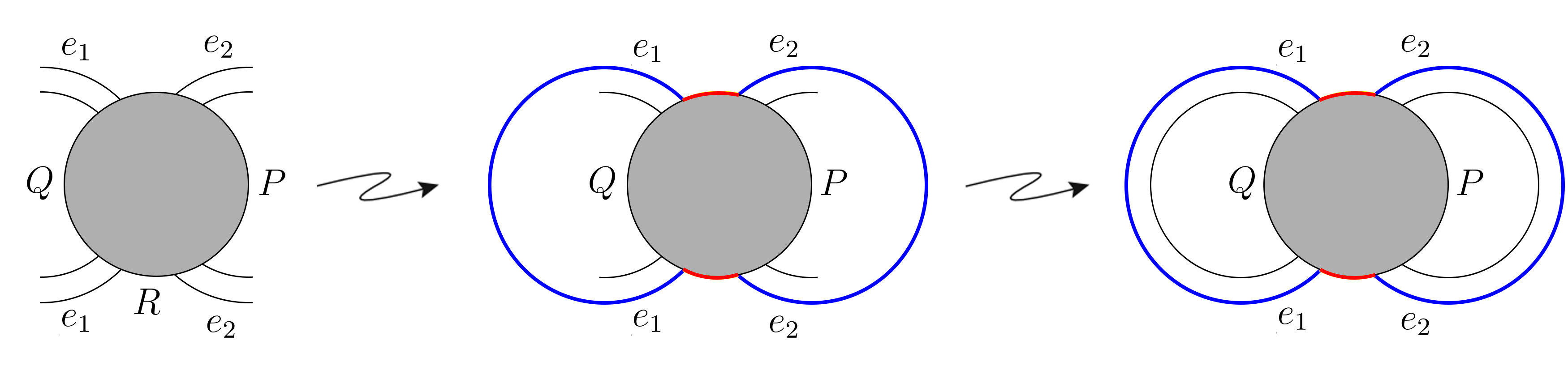}\\
  \caption{$B=[e_1, e_2, P, \pm e_2, R, \pm e_1, Q]$}\label{Fig01}
\end{figure}
According to Figure \ref{Fig01}, $R=[\emptyset]$ and $e_1$, $e_2$ must be orientable; otherwise, the boundary $\alpha$ would contain at least three edge line segments. Thus, $B=[e_1, e_2, P, e_2, e_1, Q]$. If $P=[\emptyset]$, $B$ would not be prime, contradicting the condition. Therefore, $P\neq[\emptyset]$. Assume $B=[e_1, e_2, e_3, P', e_2, e_1, Q]$. Since $e_2$ is orientable, we have that $e_2$ and $e_3$ cannot be interlaced and $B=[e_1, e_2, e_3, P'', e_3, e_2, e_1, Q]$ as shown in Figure \ref{Fig03}; otherwise, there would exist a boundary containing at least three edge line segments.
\begin{figure}[htbp]
  \centering
  \includegraphics[width=14cm]{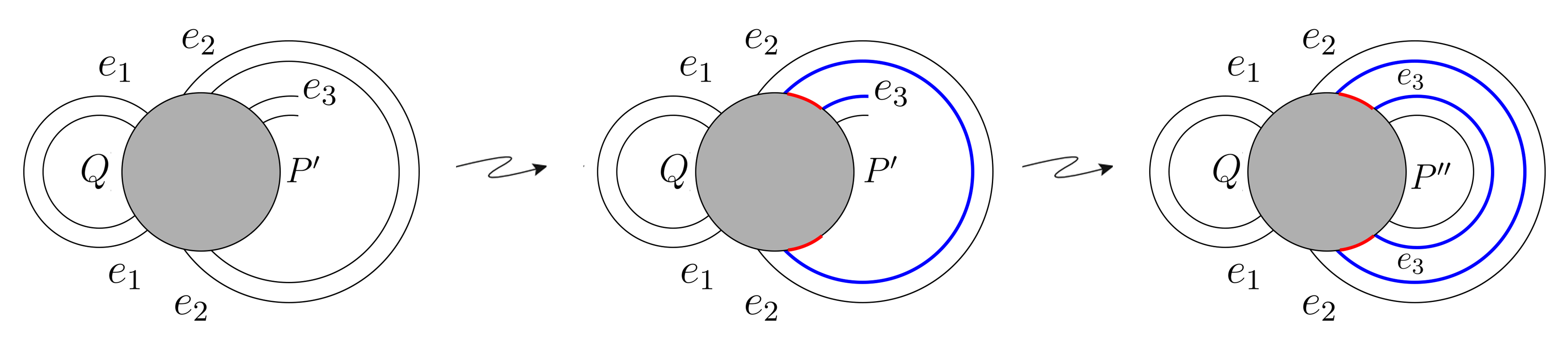}\\
  \caption{$B=[e_1, e_2, e_3, P', e_2, e_1, Q]$}\label{Fig03}
\end{figure}
Continuing this reasoning, we eventually obtain \[B = [e_1, e_2, e_3, \ldots, e_n, \overline{P}, e_n, \ldots, e_3, e_2, e_1, Q]\] with $\overline{P} \neq [\emptyset]$, which contradicts $e(B) = n$. Hence, $e_1$ and $e_2$ must be interlaced.

\end{description}
By Claim 1, assume $B=[e_1, e_2, M, \pm e_1, L, \pm e_2, N]$. As shown in Figure \ref{Fig02}, $L=[\emptyset]$ and $e_1$, $e_2$ are both non-orientable to avoid extra edge segments in boundary $\alpha$.
\begin{figure}[htbp]
  \centering
  \includegraphics[width=14cm]{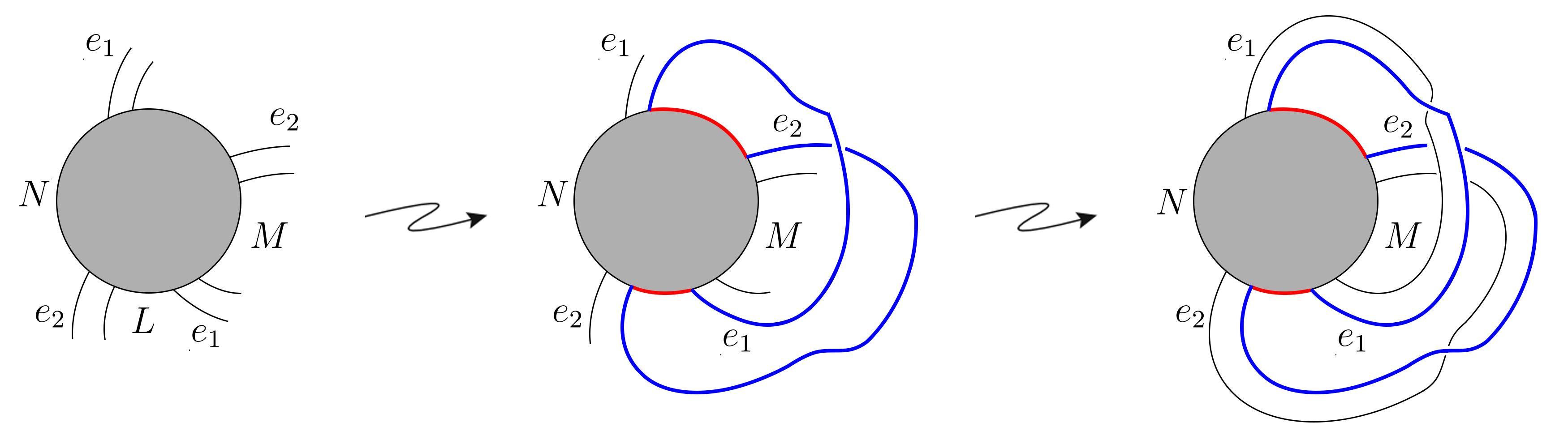}\\
  \caption{$B=[e_1, e_2, M, \pm e_1, L, \pm e_2, N]$}\label{Fig02}
\end{figure}
Thus, we conclude that $B=[e_1, e_2, M, -e_1, -e_2, N]$. If $M = \emptyset$, then primality of $B$ implies $N = \emptyset$. Otherwise, assume $B = [e_1, e_2, e_3, M', -e_1, -e_2, N]$. By Claim 1, $e_2$ and  $e_3$ are interlaced. Following the same reasoning, we derive \[B = [e_1, e_2, e_3, M', -e_1, -e_2, -e_3, N'].\] Continuing this process, if $M' = [\emptyset]$, then $N' = [\emptyset]$. Otherwise, we obtain \[B = [e_1, e_2, e_3, e_4, \ldots, e_n, -e_1, -e_2, -e_3, -e_4, \ldots, -e_n].\]
This completes the proof.

\end{proof}

\begin{theorem}\label{the-03}
Let $G$ be a connected graph with $v(G)=n$, where $n \geq 2$. Then \[P_{G}^{\times}(z) = a_1z+a_2z^2+\cdots +a_nz^n\] where $a_i\neq 0$ for all $1\leq i \leq n$, if and only if $G$ is a complete graph.
\end{theorem}

\begin{proof}
Let $B$ be a bouquet such that $G=I(B)$. We have $P_{G}^{\times}(z)={^\partial\varepsilon^{\times}_{B}(z)}$.
Since $G$ is connected, it follows that $B$ is prime.

For any $A\subseteq E(B)$, by Euler's formula, we have \[\varepsilon(B^{\times|A})=e(B^{\times|A})+2-v(B^{\times|A})-f(B^{\times|A}).\]
By Lemma \ref{lem03}, $f(B^{\times|A})\leq n$. Since $e(B^{\times|A})=e(B)=n$ and $v(B^{\times|A})=v(B)=1$, we have \[\varepsilon(B^{\times|A})=n+2-1-f(B^{\times|A})\geq 1.\]
Hence, $^\partial\varepsilon^{\times}_{B}(z)$ has degree at least $1$.

By Propositions \ref{pro02} and \ref{pro01}, \[P_{G}^{\times}(z)={^\partial\varepsilon^{\times}_{B}(z)} = a_1z+a_2z^2+\cdots +a_nz^n\] where $a_i\neq 0$ for all $1\leq i \leq n$ if and only if the minimum degree of $^\partial\varepsilon^{\times}_{B}(z)$ is 1 if and only if there exists $A'\subseteq E(B)$ such that $f(B^{\times|A'})=n$ if and only if \[B^{\times|A'}=[e_1, e_2, \ldots, e_n, -e_1, -e_2, \ldots, -e_n]\] by Lemma \ref{lem03} if and only if $I(B)$ (or $G$) is a complete graph.
\end{proof}

\begin{lemma}\label{lem-04}
Let $B_{n}$ be a bouquet with $e(B_n)=n$. If each loop in $B_n$ is trivial  and the number of orientable loops in $B_n$ is $m$, then $f(B_n)=m+1$.
\end{lemma}
\begin{proof}
Consider an edge $e \in E(B_n)$. If $e$ is a trivial orientable loop, suppose the signed rotation of $B_n$ is $[e, M, e, N]$. Then $B_n=[e, e]\vee [M]\vee [N]$. It follows that
\begin{eqnarray*}
  f(B_n)&=& f([e, e])+f([M])+f([N])-2\\
&=&f([M])+f([N])\\
&=&f([M]\vee [N])+1\\
&=&f([MN])+1.
\end{eqnarray*}
The same reasoning applies when $e$ is a trivial non-orientable loop, yielding $f(B_n)=f([MN])$. Note that each loop in $[MN]$ is also trivial. Applying the same argument iteratively to all loops, we conclude that $f(B_n)=m+1$.
\end{proof}

\begin{lemma}\label{lem-02}
Let $B_{n}$ be a bouquet with the signed rotation $$[1, 2, \cdots, n, 1, 2, \cdots, n]$$ and let $A=\{i_1, \cdots, i_k\}\subseteq [n]$ where $k\geq1$. Then, $f((B_{n})^{\times|A})=k.$
\end{lemma}
\begin{proof}
First, observe that the signed rotation of $(B_n)^{\times|A}$ is given by
\begin{gather*}
[1, 2, \cdots, i_1-1, -i_1, i_1+1, \cdots, i_k-1, -i_k, i_k+1, \cdots, n, \\
1, 2, \cdots, i_1-1, i_1, i_1+1, \cdots, i_k-1, i_k, i_k+1, \cdots, n].
\end{gather*}
Let \[P=[i_1+1, \cdots, i_k-1, -i_k, i_k+1, \cdots, n, 1, 2, \cdots, i_1-1].\] Then, we can rewrite $(B_n)^{\times|A}$ as \[[1, 2, \cdots, i_1-1, -i_1, P, i_1, i_1+1, \cdots, i_k-1, i_k, i_k+1, \cdots, n].\] By applying Operation 2 (with $a = i_1$ and $b = P$), we obtain \[(B_{n})^{\times|A}\mapsto[1, 2, \cdots, i_1-1, P^{-1}, -i_1, i_1, i_1+1, \cdots, i_k-1, i_k, i_k+1, \cdots, n].\] This result in the new bouquet $B'_n$, which has the form:
\begin{gather*}
[1, 2, \cdots, i_1-1, -(i_1-1), -(i_1-2), \cdots, -2, -1, -n, \cdots, -(i_k+1), i_k,\\ -(i_k-1), \cdots, -(i_1+1), -i_1, i_1, i_1+1, \cdots, i_k-1, i_k, i_k+1, \cdots, n].
\end{gather*}
By Proposition \ref{pro03}, we have $f((B_{n})^{\times|A})=f(B'_n)$. Since each loop in $B'_n$ is trivial and the set of orientable loops in $B'_n$ is $\{i_2, \cdots, i_k\}$, it follows form  Lemma \ref{lem-04} that \[f(B'_n)=|\{i_2, \cdots, i_k\}|+1=k.\]
Therefore, $f((B_{n})^{\times|A})=k$.
\end{proof}

\begin{lemma}[\cite{QYJ2}]\label{lem-01}
Let $B_{n}$ be a bouquet with the signed rotation $$[1, 2, \cdots, n, 1, 2, \cdots, n].$$ Then
\begin{eqnarray*}
f(B_{n})=\left\{\begin{array}{ll}
                    1, & \mbox{if}~n~\mbox{is even,}\\
                    2, & \mbox{if}~n~\mbox{is odd.}
                   \end{array}\right.
\end{eqnarray*}
\end{lemma}

\begin{theorem}
Let $K_n$ be a complete graph with $n$ vertices, where $n \geq 2$.
Then
\[P_{K_n}^{\times}(z)=
\begin{cases}
z^{n} + \sum_{i = 1}^{n} \binom{n}{n + 1 - i} z^{i}, & \text{if } n \text{ is even,} \\
z^{n - 1} + \sum_{i = 1}^{n} \binom{n}{n + 1 - i} z^{i}, & \text{if } n \text{ is odd.}
\end{cases}
\]
\end{theorem}

\begin{proof}
Let $B_{n}$ be a bouquet with the signed rotation $$[1, 2, \cdots, n, 1, 2, \cdots, n].$$
Since $I(B_n)$ is a complete graph, we have $P_{K_n}^{\times}(z)={^\partial\varepsilon^{\times}_{B_n}(z)}$.
By Lemmas \ref{lem-02} and \ref{lem-01}, in conjunction with Euler's formulas, we obtain the following two cases:

\begin{description}
  \item[Case 1: $n$ is even]
  \begin{enumerate}
    \item $\varepsilon((B_{n})^{\times|A}) = n$ if and only if $f((B_{n})^{\times|A}) = 1$, which holds true when $|A| = 0$ or $1$. In this case, $A$ can be chosen in $\binom{n}{0} + \binom{n}{1} = 1 + n$ ways.
    \item $\varepsilon((B_{n})^{\times|A}) = i$ (for $1 \leq i \leq n-1$) if and only if $f((B_{n})^{\times|A}) = n + 1 - i$, which occurs when $|A| = n + 1 - i$. Here, $A$ can be chosen in $\binom{n}{n + 1 - i}$ ways.
  \end{enumerate}

  \item[Case 2: $n$ is odd]
  \begin{enumerate}
    \item $\varepsilon((B_{n})^{\times|A}) = n - 1$ if and only if $f((B_{n})^{\times|A}) = 2$, which is valid when $|A| = 0$ or $2$. In this scenario, $A$ can be chosen in $\binom{n}{0} + \binom{n}{2} = 1 +\binom{n}{2}$ ways.
    \item $\varepsilon((B_{n})^{\times|A}) = i$ (for $i = n$ or $1 \leq i \leq n - 2$) if and only if $f((B_{n})^{\times|A}) = n + 1 - i$, which holds when $|A| = n + 1 - i$. In such cases, $A$ can be chosen in $\binom{n}{n + 1 - i}$ ways.
  \end{enumerate}
\end{description}

\end{proof}

\section{Partial Petrial polynomials for paths}

Let $F_n~(n\geq 0)$ be a bouquet with the signed rotation
\[
[e_1, e_2, e_1, e_3, e_2, \ldots, e_{n-1}, e_{n-2}, e_n, e_{n-1}, e_n].
\] Figure \ref{Fig04} shows the bouquets $F_{6}$.
\begin{figure}[htbp]
  \centering
  \includegraphics[width=5cm]{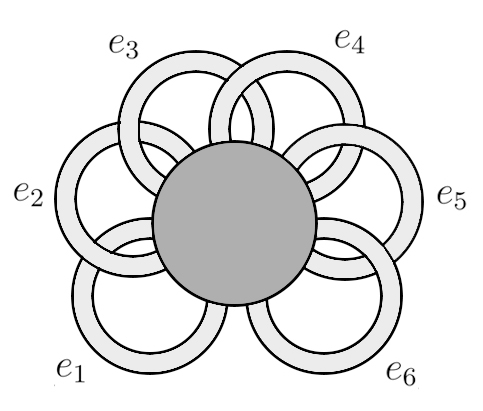}\\
  \caption{The bouquet $F_{6}$.}\label{Fig04}
\end{figure}

\begin{lemma}\label{lem-05}
Let $B_n$ be a partial Petrial of  $F_n$. Then $B_n$ can be transformed into either an isolated vertex or a vertex with exactly one orientable loop by a sequence of operations 1-4.

\end{lemma}

\begin{proof}
We prove the lemma by induction on $n$.

\noindent{\bf Basis Step ($n \leq 1$)}:
Note that $B_n$ can be $[\emptyset]$, $[e_1, e_1]$, or $[-e_1, e_1]$.
\begin{itemize}
  \item If $B_n = [\emptyset]$, then $B_n$ is an isolated vertex.
  \item If $B_n = [e_1, e_1]$, then $B_n$ is a vertex with exactly one orientable loop.
  \item If $B_n = [-e_1, e_1]$, then applying Operation 4 with $a = e_1$ transforms $B_n$ into $[\emptyset]$, hence an isolated vertex.
\end{itemize}

\noindent{\bf Induction step ($n>1$)}: We consider the following three cases:
\begin{description}
  \item[Case 1] $B_n=[-e_1, e_2, e_1, M]$

$B_n=[-e_1, e_2, e_1, M]\xrightarrow[a=e_1,~b=e_2]{\mbox{\footnotesize Operation 2}}[-e_2, -e_1, e_1, M]\xrightarrow[a=e_1]{\mbox{\footnotesize Operation 4}}[-e_2, M]$.
  \item[Case 2] $B_n=[e_1, e_2, e_1, e_3, e_2, M]$

$B_n=[e_1, e_2, e_1, e_3, e_2, M]\xrightarrow[a=e_1,~b=e_3]{\mbox{\footnotesize Operation 1}}[e_3, e_1, e_2, e_1, e_2, M]\xrightarrow[a=e_1,~b=e_2]{\mbox{\footnotesize Operation 3}}[e_3, M]$.
  \item[Case 3] $B_n=[e_1, -e_2, e_1, e_3, e_2, M]$
\begin{gather*}
B_n=[e_1, -e_2, e_1, e_3, e_2, M]\xrightarrow[a=e_1,~b=e_3]{\mbox{\footnotesize Operation 1}}[e_3, e_1, -e_2, e_1, e_2, M]\xrightarrow[a=e_2,~b=e_1]{\mbox{\footnotesize Operation 2}}~~~~~~~\\
[e_3, e_1, -e_1, -e_2, e_2, M]\xrightarrow[a=e_1]{\mbox{\footnotesize Operation 4}}[e_3, -e_2, e_2, M]\xrightarrow[a=e_2]{\mbox{\footnotesize Operation 4}}[e_3, M].~~~~~~~
\end{gather*}
\end{description}
Note that $[-e_2, M]$ and $[e_3, M]$ are partial Petrials of  $F_{n-1}$ and $F_{n-2}$, respectively. By the induction hypothesis, $[-e_2, M]$ and $[e_3, M]$ can be transformed into either an isolated vertex or a vertex with exactly one orientable loop using a sequence of operations 1-4. Therefore, the same conclusion holds for $B_n$.
\end{proof}

\begin{theorem}\label{the-02}
Let $P_n$ be a path with $n$ vertices, where $n \geq 1$.
Then
\begin{equation*}
P_{P_n}^{\times}(z) =
\begin{cases}
\left(\frac{2^{n}-1}{3}\right)z^{n-1} + \left(\frac{2^{n+1}+1}{3}\right)z^n, & \text{if } n \text{ is even,}\\
\left(\frac{2^{n}+1}{3}\right)z^{n-1} + \left(\frac{2^{n+1}-1}{3}\right)z^n, & \text{if } n \text{ is odd.}
\end{cases}
\end{equation*}
\end{theorem}
\begin{proof}
Since $I(F_n)$ is a path, we have $P_{P_n}^{\times}(z) = {^\partial\varepsilon^{\times}_{F_n}(z)}$.
Let $\mathbb{F}_n$ be the set of all partial Petries of $F_n$. For any $B_n \in \mathbb{F}_n$, by Proposition \ref{pro03} and Lemma \ref{lem-05}, $f(B_n) = f([\emptyset]) = 1$ or $f(B_n) = f([e_n, e_n]) = 2$. By Euler's formula, $\varepsilon(B_n) = n$ or $\varepsilon(B_n) = n-1$. Hence, the polynomial $^\partial\varepsilon^{\times}_{F_n}(z)$ has at most two terms. In the following, we will show that the polynomial has exactly two terms.

Let $\mathbb{F}_{n-1} = \{W_n - e_n \mid W_n \in \mathbb{F}_n\}$. For any $B_{n-1} \in \mathbb{F}_{n-1}$, there exist exactly two bouquets $B_n', B_n'' \in \mathbb{F}_n$ such that $B_n'= (B_n'')^{\times|e_n}$ and $B_{n-1} = B_n' - e_n = B_n'' - e_n$. Assume by symmetry that $e_n$ in $B_n'$ is orientable, and in $B_n''$ is non-orientable. By Lemma \ref{lem-05}, we have the following three cases:

\begin{description}
  \item[Case 1] $B_{n-1}\longmapsto [e_{n-2}, e_{n-1}, e_{n-2}, e_{n-1}] \longmapsto [\emptyset]$

$B_n'\longmapsto[e_{n-2}, e_{n-1}, e_{n-2}, e_{n}, e_{n-1}, e_{n}] \longmapsto [e_{n}, e_{n}]$

$B_n''\longmapsto[e_{n-2}, e_{n-1}, e_{n-2}, -e_{n}, e_{n-1}, e_{n}] \longmapsto [-e_{n}, e_{n}]\longmapsto [\emptyset]$.

  \item[Case 2] $B_{n-1}\longmapsto [e_{n-1}, -e_{n-1}] \longmapsto [\emptyset]$

$B_n'\longmapsto[e_{n-1}, e_{n}, -e_{n-1}, e_{n}]\longmapsto [e_{n}, -e_{n}]\longmapsto [\emptyset]$

$B_n''\longmapsto[e_{n-1}, -e_{n}, -e_{n-1}, e_{n}]\longmapsto [e_{n}, e_{n}]$.

  \item[Case 3] $B_{n-1}\longmapsto [e_{n-1}, e_{n-1}]$

$B_n'\longmapsto[e_{n-1}, e_{n}, e_{n-1}, e_{n}]\longmapsto [\emptyset]$

$B_n''\longmapsto[e_{n-1}, -e_{n}, e_{n-1}, e_{n}]\longmapsto [\emptyset]$.
\end{description}

Let $\{\mathbb{F}_n\longmapsto [\emptyset]\}$ and $\{\mathbb{F}_n\longmapsto [e_n, e_n]\}$ be the sets of bouquets in $\mathbb{F}_n$ that can change into $[\emptyset]$ and $[e_n, e_n]$, respectively. By the above three cases, there is a one-to-one correspondence between $\{\mathbb{F}_{n-1}\longmapsto [\emptyset]\}$ and $\{\mathbb{F}_n\longmapsto [e_n, e_n]\}$.
Hence, \[|\{\mathbb{F}_{n-1}\longmapsto [\emptyset]\}|=|\{\mathbb{F}_n\longmapsto [e_n, e_n]\}|.\]
Therefore,
\begin{eqnarray*}
^\partial\varepsilon^{\times}_{F_1}(z)&=&a_{0}+a_1z~=~1+z,\\
^\partial\varepsilon^{\times}_{F_2}(z)&=&a_{1}z+a_2z^2,\\
\vdots\\
^\partial\varepsilon^{\times}_{F_{n-1}}(z)&=&a_{n-2}z^{n-2}+a_{n-1}z^{n-1},\\
^\partial\varepsilon^{\times}_{F_n}(z)&=&a_{n-1}z^{n-1}+a_nz^n,
\end{eqnarray*}
Since $^\partial\varepsilon^{\times}_{F_n}(1)=2^n$, we deduce that the sequence $\{a_n\}$ satisfies the recursion relation: $a_0=1$ and $a_{n-1}+a_n=2^n$.
Using this recursion relation, we obtain
\begin{equation*}
P_{P_n}^{\times}(z)={^\partial\varepsilon^{\times}_{F_n}(z)} =
\begin{cases}
\left(\frac{2^{n}-1}{3}\right)z^{n-1} + \left(\frac{2^{n+1}+1}{3}\right)z^n, & \text{if } n \text{ is even,}\\
\left(\frac{2^{n}+1}{3}\right)z^{n-1} + \left(\frac{2^{n+1}-1}{3}\right)z^n, & \text{if } n \text{ is odd.}
\end{cases}
\end{equation*}
\end{proof}

\section{Concluding remarks}
In Theorem \ref{the-03}, we prove that for a connected graph with $n$ vertices ($n \geq 2$), the partial Petrial polynomial has non-zero coefficients for all terms of degrees from 1 to $n$ if and only if the graph is complete. Furthermore, Theorem \ref{the-02} establishes that the partial Petrial polynomial of a path is a binomial. This raises the question of whether the converse holds:
\begin{problem}
If the partial Petrial polynomial of a connected graph $G$ is a binomial, must $G$ necessarily be a path? If not, can we characterize the graphs whose partial Petrial polynomials are binomials?
\end{problem}

\section*{Acknowledgements}
This work is supported by NSFC (Nos. 12471326, 12101600).



\begin{thebibliography}{9}
\bibitem{BR} B. Bollob\'as and O. Riordan, A polynomial of graphs on surfaces,
\emph{Math. Ann.} \textbf{323} (2002) 81--96.

\bibitem{CSL} S. Chmutov and S. Lando, Mutant knots and intersection graphs, \emph{Algebr. Geom. Topol.} \textbf{7} (2007) 1579--1598.

\bibitem{EM1}  J. A. Ellis-Monaghan and I. Moffatt, Twisted duality for embedded graphs, \emph{Trans. Amer. Math. Soc.} \textbf{364} (2012) 1529--1569.

\bibitem{EM}  J. A. Ellis-Monaghan and I. Moffatt, Graphs on surfaces, Springer New York, 2013.


\bibitem{GT} J. L. Gross and T. W. Tucker, Topological Graph Theory, John Wiley \&
Sons, New York, 1987.


\bibitem{GMT2} J. L. Gross, T. Mansour and T. W. Tucker, Partial duality for ribbon graphs, II: Partial-twuality polynomials and monodromy computations, \emph{European J. Combin.} \textbf{95} (2021) 103329.

\bibitem{Metrose}  M. Metsidik,Characterization of some properties of ribbon graphs and their partial duals, PhD thesis, Xiamen University, 2017.

\bibitem{Moffatt15} I. Moffatt, Separability and the genus of a partial dual, \emph{European J. Combin.} \textbf{34} (2013) 355--378.

\bibitem{SEW} S. E. Wilson, Operators over regular maps, \emph{Pacific J. Math.} \textbf{81} (1979) 559--568.

\bibitem{QYJ3} Q. Yan and X. Jin, Partial-dual genus polynomials and signed intersection graphs, \emph{Forum Math. Sigma} \textbf{10} (2022) e69.


\bibitem{QYJ} Q. Yan and X. Jin, A-trails of embedded graphs and twisted duals, \emph{Ars Math. Contemp.} \textbf{22} (2022) \# P2.06.

\bibitem{QYJ2} Q. Yan and X. Jin, Counterexamples to a conjecture by Gross, Mansour and Tucker on partial-dual genus polynomials of ribbon graphs, \emph{European J. Combin.} \textbf{93} (2021) 103285.
\end{thebibliography}
\end{document}